\documentclass{amsart}
\usepackage{amsmath,amsfonts}

\newtheorem{theorem}{Theorem}

\newtheorem{lemma}[theorem]{Lemma}

\newtheorem{remark}{Remark}

\newcommand{\Pp}{\mathbf{P}}
\newcommand{\E}{\mathbf{E}}
\newcommand{\eps}{\epsilon}

\newcommand{\R}{\mathbb{R}}

\newcommand{\Span}{\mathop{\rm span}}
\newcommand{\one}{\mathbf{1}}

\newcommand{\Fc}{\mathcal{F}}

\title{Scaling Limit for the Diffusion Exit Problem in the Levinson Case}
\author{Sergio Angel Almada Monter \and Yuri Bakhtin }
\address{School of Mathematics, Georgia Tech, Atlanta GA, 30332-0160, USA}
\email{salmada3@math.gatech.edu, bakhtin@math.gatech.edu}
\begin{document}

\maketitle

\begin{abstract}
The exit problem for small perturbations of a dynamical system in a domain is considered. It is assumed that
the unperturbed dynamical system and the domain satisfy the Levinson conditions. We assume that the
random perturbation affects the driving vector field and the initial condition, and each of the components
of the perturbation follows a scaling limit. We derive the joint scaling limit for the random exit time and exit point.
We use this result to study the asymptotics of the exit time for 1-d diffusions conditioned on rare events.
\end{abstract}
%
\section{Introduction}
\label{Sec: Notation}
Small stochastic perturbations of deterministic dynamical systems
have been studied for several decades, see, e.g., a set of lectures~\cite{Freidlin-lectures:MR1399081} and references therein. 
Properties of exit distributions for the resulting diffusions are particularly important. One reason for that is that 
one can express the solutions of parabolic and elliptic PDE's containing the generator of the diffusion via the exit 
distributions. Another reason is the possibility to use exit distributions for the analysis of the global behavior of the system.
One can cover the state space by several domains and study the process within each of them separately. Using the strong Markov property one
can then treat the exit distribution for one of the domains as the starting distribution for the next one. This approach was vital for the study
of noisy heteroclinic networks in the vanishing noise limit, see~\cite{nhn},\cite{nhn-ds}.

In this note we study a relatively simple situation called the Levinson case (see~\cite[Chapter 2]{Freidlin-lectures:MR1399081}),
where the typical exit happens along a trajectory of the deterministic flow.
 We derive a scaling law for the exit
distribution in the limit of vanishing perturbation assuming that the initial random data as well as both
deterministic and white noise components of the perturbation follow a scaling limit. 

We also show that our main result can be used to study rare events for diffusions. We present a 1-dimensional
situation where to reach a certain threshold, the diffusion has to evolve against the deterministic flow. By conditioning
on this unlikely event, we reduce the analysis to the Levinson case.

The paper is organized as follows. In Section~\ref{sec:main-results} we state the main theorem for the Levinson case,
postponing its proof to Section~\ref{sec:proof-levinson}.
In Section~\ref{sec:1-d-diffusion} we state the result on the diffusion conditioned on a rare event and derive it
from the main theorem and some auxiliary statements proven in Section~\ref{sec:aux-1-d}.

\medskip

{\bf Acknowledgment.} Yuri Bakhtin is grateful to NSF for partial support via a CAREER grant DMS-0742424.  

\section{Main result}\label{sec:main-results}

We consider a $C^2$-smooth bounded vector field~$b$ in $\R^d$. The unperturbed dynamics is given by
 the deterministic flow $S=(S^t)_{t\in\R}$ generated by $b$: 
\[
\frac{d}{dt}S^{t}x_0=b(S^{t}x_0),\quad S^{0}x_0=x_0.
\]

We will introduce three components of perturbations of this deterministic flow. They all depend on
a small parameter $\eps>0$. 

The first component is white noise perturbation generated by $\eps\sigma$, where
$\sigma:\R^d \to \R^{d \times d}$ is a $C^2$-smooth bounded matrix valued function.

The second one is $\eps^{\alpha_1}\Psi_\eps$, where $\Psi_\eps$ is
a deterministic Lipschitz vector field on $\R^d$ for each $\eps$,  converging uniformly to 
a limiting Lipschitz vector field $\Psi_0$, and $\alpha_1$ is a positive scaling exponent.
These conditions ensure that the stochastic It\^o
equation  
\begin{equation}
dX_\eps(t) = \left( b( X_\eps(t) )+\eps^{\alpha_1} \Psi_\eps(X_\eps(t)) \right)dt + \eps \sigma(X_\eps(t))dW \label{eq:PrincipalEquation}
\end{equation}
w.r.t.\ a standard $d$-dimensional Wiener process $W$ %
has a unique strong solution for any $\eps>0$ and all initial conditions (for a general background on stochastic differential equations see, e.g.,~\cite{Karatzas--Shreve}).  

The last component of the perturbation is the initial condition satisfying
\begin{equation}
X_{\epsilon }(0)=x_{0}+\epsilon ^{\alpha_2 }\xi _{\epsilon },\quad\eps>0.
\label{eq:initial_condition}
\end{equation}%
Here  $\alpha_2>0$, and $(\xi _{\epsilon })_{\epsilon >0}$ is
a family of random variables independent of $W$, such that for some random
variable $\xi _{0}$, $\xi _{\epsilon }\rightarrow \xi _{0}$ as $\epsilon
\rightarrow 0$ in distribution.
 
\medskip

Let $M$ be a smooth $C^2$-hypersurface in $\R^d$.  If 
\begin{equation*}
\tau _{\epsilon }=\inf \left \{t\ge 0:X_{\epsilon
}(t)\in M \right \},
\end{equation*}
then on $\{\tau_\eps<\infty\}$ we have $X_\eps(\tau_\eps)\in M$. We are going to study
the limiting behavior of $\tau_\eps$ and $X_\eps(\tau_\eps)$ as $\eps\to 0$ 
under the assumptions above.

Let us describe our assumptions on the joint geometry of the vector field $b$ and the surface $M$. 
First we define
\[
T=\inf \left \{t>0: S^tx_0 \in M \right \},
\]
and assume that $0<T<\infty$. Secondly, we denote $z=S^Tx_0\in M$ and assume that 
$b(z)$ does not belong to the tangent
hyperplane $T_zM$. In other words, we assume that the positive orbit of $x_0$ intersects $M$ and
the crossing is transversal. 

In the case of $\xi_\eps\equiv 0$ and $\Psi \equiv 0$, Levinson's theorem states (see \cite{Levinson:MR0033433}, 
\cite[Chapter 2]{Freidlin--Wentzell-book}, and \cite[Chapter 2]{Freidlin-lectures:MR1399081}) that 
$X_\eps(\tau_\eps)\to z$ in probability as $\eps \to 0$. Levinson worked in the PDE context and 
showed how to obtain an expansion for the solution of the corresponding elliptic PDE depending 
on the small parameter~$\eps$.  
The main result of this note describes the limiting behavior of the correction $(\tau_\eps-T,X_\eps(\tau_\eps)-z)$ and extends  \cite[Theorem 2.3]{Freidlin--Wentzell-book} to the situation with generic
perturbation parameters $\xi_0,\Psi,\alpha_1$, and $\alpha_2$. This extension is essential since, as the analysis in~\cite{nhn}
shows, in the sequential study of entrance-exit distributions for multiple domains one has to consider nontrivial scaling laws for the initial conditions; also, considering nontrivial deterministic perturbations will allow us to study rare events, see 
Section~\ref{sec:1-d-diffusion}.

We need more notation. Due to the smoothness of $b$,
\begin{equation}
b(x)=b(y)+Db(y)(x-y)+Q_1(y,x-y),\quad x,y\in\R^d, \label{eqn: b_Taylor}
\end{equation} 
where 
\begin{equation}
|Q_1(u,v)|\leq K  |v|^2 \label{eqn: bound_Q1},
\end{equation}
 for some constant $K>0$ and any $u,v \in \R^d$.
We denote by $\Phi_x(t)$ the linearization of $S$ along the
orbit of $x$:
\begin{equation}
\frac{d}{dt}\Phi_x(t)=A(t)\Phi_x(t)\text{, \ }\Phi_x(0)=I, \label{eqn: Phi_def}
\end{equation}
where $A(t)=Db(S^tx)$ and $I$ is the identity matrix. 
 
Finally, for any vector $v\in \R^d$, we define $\pi_b v\in \R$ and $\pi_M v\in T_zM$ by
\[
 v=\pi_b v \cdot b(z)+ \pi_M v,
\]
i.e., $\pi_b$ is the (algebraic) projection onto $\Span(b(z))$ along $T_zM$ and 
$\pi_M$ is the (geometric) projection onto $T_zM$ along $\Span(b(z))$.

 \begin{theorem}
 \label{thm: Main}  Let $\alpha=\alpha_1 \wedge \alpha_2\wedge 1$, and
\begin{align} \notag
\phi_0 (t) &= \mathbf{1}_{\{\alpha_2= \alpha\}}  \Phi_{x_0}(t) \xi_0 + \mathbf{1}_{\{\alpha_1= \alpha\}}\Phi_{x_0}(t) \int_0^t \Phi_{x_0} (s)^{-1} \Psi_0( S^s x) ds \\
& \quad +\mathbf{1}_{\{1=\alpha\}} \Phi_{x_0}(t)\int_{0}^{t}\Phi_{x_0}^{-1}(s)\sigma(S^s x_0)dW(s), \quad t>0.
\label{eqn: phi_0_def}
\end{align}
Then, in the setting introduced above, 
\begin{equation}
 \eps^{-\alpha}(\tau_\eps-T, X_\eps(\tau_\eps)-z){\to} (-\pi_b \phi_0(T), \pi_M \phi_0(T)). \label{eq:main_convergence_statement}
\end{equation}
in distribution.
If additionally we require that $\xi_\eps\to\xi_0$ in probability or that $\alpha_2>\alpha$, then the convergence in \eqref{eq:main_convergence_statement}
is also in probability.
\end{theorem}

\begin{remark}\label{rem:weakening-conditions-by-localization}
\rm The conditions of Theorem~\ref{thm: Main} can be relaxed using the standard localization procedure. In fact, one needs to
require uniform convergence of $\Psi_\eps\to\Psi_0$ and regularity properties of $b$ and $\sigma$
only in some neighborhood of the set $\{S^tx_0:\ 0\le t\le T(x_0)\}$.
\end{remark}

\begin{remark}\rm In applications (see~\cite{nhn},\cite{nhn-ds}), the parameters $\alpha_1$ and $\alpha_2$ can be chosen so that
the r.h.s.\ of~\eqref{eq:main_convergence_statement} is nondegenerate. 
\end{remark}

\begin{remark}\rm
In the case where $d=1$, the hypersurface $M$ is just a point. Therefore, $\pi_M$ is identical zero and the only 
contentful information Theorem~\ref{thm: Main} provides is the asymptotics of the exit time. 
\end{remark}

\section{Conditioned diffusions in 1 dimension}\label{sec:1-d-diffusion}

In this section we apply Theorem~\ref{thm: Main} to the analysis of the exit time of conditioned diffusions in 1-dimensional situation for the large deviation case.

Suppose, for each $\eps>0$, $X_\eps$ is a weak solution of the following SDE:
\begin{align*}
 dX_\eps(t)&=b(X_\eps(t))dt+\eps\sigma(X_\eps(t)) dW(t),\\
 X_\eps(0)&=x_0,
\end{align*}
where $b$ and $\sigma$ are $C^1$ functions on $\R$, such that $b(x)<0$ and $\sigma(x)\ne 0$ for all $x$ in an interval $[a_1,a_2]$ 
containing $x_0$. We introduce 
\[
 \tau_\eps=\inf\{t\ge0:\ X_\eps(t)=a_1\ \text{\rm or}\ a_2\}
\]
and  $B_\eps=\{X_\eps(\tau_\eps)=a_2\}$. Since $b<0$, $B_\eps$ is a rare event since $\lim_{\eps\to0}\Pp(B_\eps)=0$. More precise estimates
on the asymptotic behavior of $\Pp(B_\eps)$ can be obtained in terms of large deviations. However, here we study the diffusion
$X_\eps$ conditioned on the rare event $B_\eps$.

Let $T(x_0)$ denote the time it takes for the solution of $\dot x=-b(x)$ starting at $x_0$ to reach $a_2$: 
\[
T(x_0)=-\int_{x_0}^{a_2}\frac{1}{b(x)}dx.
\]

\begin{theorem} \label{thm: conditioned} Conditioned on $B_\eps$, the distribution of $\eps^{-1}(\tau_\eps-T(x_0))$ converges weakly to a centered 
Gaussian
distribution with variance 
\[
- \int_{x_0}^{a_2} \frac{ \sigma^2 (y) }{ b^3 (y) } dy.
\]
\end{theorem}

To prove this theorem, we will need two auxiliary statements. Their proofs are given in Section~\ref{sec:aux-1-d}.

\begin{lemma} \label{lemma: conditioned_drift}Conditioned on $B_\eps$, the process $X_\eps$ is a diffusion with the same
diffusion coefficient as the unconditioned process, and with the
drift coefficient given by 
\[
b_\eps(x)= b(x)+\eps^2\sigma^2(x)\frac{h_\eps(x)}{\int_{a_1}^x h_\eps(y)dy},
\]
where
\begin{equation}
 \label{eq:h_eps}
 h_\eps(x)=\exp\left\{-\frac{2}{\eps^2}\int_{a_1}^x \frac{b(y)}{\sigma^2(y)}dy\right\}.
\end{equation}
\end{lemma}

Further analysis requires understanding the limiting behavior of $b_\eps$. This is the purpose of the next lemma:
\begin{lemma} \label{lemma: b_approx} There is $\delta>0$ such that
\[
\limsup_{\eps \to 0}  \eps^{-2}\left(\sup_{x \in [x_0-\delta, a_2+\delta]} | b_\eps(x) + b (x) | \right) < \infty.
\]
\end{lemma}

\begin{proof}[Proof of Theorem~\ref{thm: conditioned}]
Let us fix $\beta \in (1,2)$. Lemmas~\ref{lemma: conditioned_drift} and~\ref{lemma: b_approx} imply that
$X_\eps$ conditioned on $B_\eps$, up to $\tau_\eps$ satisfies an SDE of the form 
\[
d X_\eps (t)= \left( -b(X_\eps(t)) + \eps^{\beta} \Psi_{\eps,\beta} ( X_\eps (t) ) \right) dt + \eps \sigma ( X_\eps (t) ) d\tilde W(t),
\]
for some Brownian Motion $\tilde W$ and with $\Psi_{\eps, \beta} \to 0$ uniformly as $\eps \to 0$. We can assume that after time $\tau_\eps$,
this process still follows the same equation at least up to the time it hits $x_0 -\delta$ or $a_1+\delta$.

So, we can apply Theorem~\ref{thm: Main}  (taking into account Remark~\ref{rem:weakening-conditions-by-localization}) to see that 
\begin{equation} \label{eqn: conv_example}
\eps^{-1} ( \tau_\eps - T(x_0) ) \stackrel{\Pp}{\longrightarrow} 
 - \frac{1}{b(a_2)}
\Phi_{x_0}(T( x_0) )\int_{0}^{T (x_0)}\Phi_{x_0}^{-1}(s)\sigma(S^s x_0)d\tilde W(s), \quad \eps\to0,
\end{equation}
where $S^t x_0$ is the flow generated by the vector field $-b$, the time $T(x_0)$ solves $S^{T(x_0)}x_0=a_2$, and
$\Phi_{x_0}$ is the linearization of $S$ near the orbit of $x_0$.
The limit is clearly a centered Gaussian random variable. To compute its variance we must first solve
\[
\frac{d}{dt}\Phi_{x_0} (t) = - b^\prime (S^t x_0 ) \Phi_{x_0} (t), \quad \Phi_{x_0} (0)=1.
\]
 The solution to this linear ODE is 
\[
\Phi_{x_0} (t)= \exp \left \{- \int_0^t b^\prime (S^s x_0 )ds \right \},
\]
so that after the change of variables $u=S^s x_0$ in the integral, we get
\[
\Phi_{x_0} (t)= \frac{b(S^t x_0)}{b(x_0)}.
\]
Using this expression and It\^o isometry for the limiting random variable in~\eqref{eqn: conv_example}, we get that the variance of such random variable is
\[
\int_0^{T(x_0)} \frac{ \sigma^2 ( S^tx_0 ) }{ b^2 ( S^t x_0 ) } dt.
\]
We can now use the change of variable $u=S^s x_0$ to get the expression in Theorem~\ref{thm: conditioned}.
\end{proof}

\section{Proof of the main result}\label{sec:proof-levinson}

With high probability, at time $T$ the process $X_\eps$ is close to $z$ and the hitting
time~$\tau_\eps$ is close to $T$. The idea of the proof is that while the diffusion is close to~$z$, the process
may be approximated very well by motion with constant velocity~$b(z)$. 

We start with a lemma and postpone its proof to the end of this section to keep continuity of the text.

\begin{lemma} 
\label{lemma: Bakhtin}
Let $X_\eps$ be the solution of the SDE~\eqref {eq:PrincipalEquation} with initial condition~\eqref{eq:initial_condition}. 
Let
\begin{align}
\Theta_\eps (t) &=\eps^{\alpha_2-\alpha}\Phi_{x_0}(t) \xi_\eps+  \eps^{\alpha_1-\alpha}\Phi_{x_0}(t) \int_0^t \Phi_{x_0} (s)^{-1} \Psi_0( S^s x_0) ds \notag \\
&\quad + \eps^{1-\alpha} \Phi_{x_0}(t) \int_0^t \Phi_{x_0} (s)^{-1}\sigma( S^s x_0) dW(s).  \label{eqn: SDE_Theta} 
\end{align} Then, 
\[
X_\eps(t)=S^tx_0 + \eps^\alpha \phi_\eps (t)
\]
holds almost surely for every $t>0$,
where $\phi_\eps(t)=\Theta_\eps(t)+ r_\eps(t)$, and $r_\eps$ converges to 0 uniformly over compact time intervals in probability.

If $\xi_\eps\to\xi_0$ in distribution, then for any $T>0$, 
$\phi_\eps \to \phi_0$  in distribution in $C[0,T]$ equipped with uniform norm,
 where $\phi_0$ is the stochastic process defined in~\eqref{eqn: phi_0_def}.

If $\xi_\eps\to\xi_0$ in probability or $\alpha_2>\alpha$, then  the uniform convergence for $\phi_\eps$
also holds in probability.
\end{lemma}

\begin{remark}\rm 
This lemma gives the first-order approximation for $X_\eps(t)$. Higher-order approximations in the spirit of~\cite{Blagoveschenskii:MR0139204}
are also possible. They can be used to refine Theorem~\ref{thm: Main}.
\end{remark}

Our task now is to analyze the process $X_\eps(t)-z$ for $t$ close to $T$. Let us first estimate the deviation of the flow $S$
from the motion with costant velocity~$b(z)$. Let 
\begin{equation}
r_\pm (t,x)= S^{\pm t}x-\left( x\pm tb(z) \right),\quad t>0,\ x\in\R^d. \label{eqn: def_r_pm}
\end{equation}
\begin{lemma}
\label{lemma: rectification}
There are constants $C_1$ and $C_2$ so that for any $t>0$ and $x\in\R^d$
\[
\sup_{s\leq t}\left|  r_\pm(s,x) \right| \leq C_1 e^{C_2 t }( t |x-z | + t^2). 
\]
\end{lemma}

\begin{proof}
We prove the result for $r_+$. The analysis of $r_-$ is similar since $S^{-t}x$ is the solution to the ODE
\[
\frac{d}{dt} S^{-t}x = -b (S^{-t}x ).
\]
Let $L>0$ be the Lipschitz constant of $b$. The proof follows from the inequalities:
\begin{align*}
\left| r_+(t,x) \right| &\leq \int_0^t \left| b(S^sx)-b(z)\right|ds \\
& \leq L \int_0^t \left| S^sx-z\right|ds \\
& \leq L\int_0^t \left|r_+ (s,x) \right|ds + L\int_0^t \left | x + sb(z)-z\right| ds \\
& \leq L\int_0^t \left|r_+ (s,x) \right|ds + L\int_0^t |x-z|ds + L\int_0^t s|b(z)| ds \\
&\leq  L\int_0^t \left| r_+(s,x) \right|ds + L t |x-z| + t^2 L|b(z)|/2.
\end{align*}
The result follows as an application of Gronwall's lemma.
\end{proof}

\begin{lemma} \label{lemma: X_before_after}
Let $\gamma \in (\alpha/2, \alpha)$. Then, there are two a.s.-continuous stochastic processes $\Gamma_{\eps,\pm}$ such that
\[
\sup_{t\in[0,\eps^\gamma]} |\Gamma_{\eps,\pm} (t)| \stackrel{\Pp}{\longrightarrow} 0, \quad \eps \to 0,
\]
and almost surely for any $t\in[0,\eps^\gamma]$
\begin{equation}
X_\eps (T-t)=z-tb(z)+\eps^\alpha \left( \phi_\eps (T-t) +\Gamma_{\eps,-}(t) \right) \label{eqn: X_before}
\end{equation} and 
\begin{equation}
X_\eps (T+t)=z+tb(z)+\eps^\alpha \left( \Phi_z (t) \phi_\eps (T) +\Gamma_{\eps,+}(t) \right). \label{eqn: X_after}
\end{equation}
\end{lemma}
\begin{proof}
Due to Lemma~\ref{lemma: Bakhtin},  the flow property, and~\eqref{eqn: def_r_pm} we have 
\begin{align*}
X_\eps (T-t) &= S^{T-t}x_0 + \eps^\alpha \phi_\eps (T-t) \\
&= S^{-t}z + \eps^\alpha \phi_\eps (T-t) \\
&=z-tb(z) + r_-(t,z)  + \eps^\alpha \phi_\eps (T-t).
\end{align*}
The first estimate with $\Gamma_{\eps,-} (t) = \eps^{-\alpha} r_-(t,z)$ follows from Lemma~\ref{lemma: rectification} for $x=z$.

Due to Strong Markov property and Lemma~\ref{lemma: Bakhtin} the process $\tilde X_\eps (t) = X_\eps (t + T)$ is a solution of the initial value problem
\begin{align*}
d\tilde X_\eps(t) &= (b( \tilde X_\eps(t))+\eps^{\alpha_1}\Psi_\eps(\tilde X_\eps(t)))dt + \eps \sigma (\tilde X_\eps(t)) d\tilde W, \\
\tilde X_\eps (0) &= X_\eps (T)=z+\eps^\alpha \phi_\eps (T),
\end{align*}
with respect to the Brownian Motion $\tilde W(t)=W(t+T)-W(T)$. So, again, applying Lemma~\ref{lemma: Bakhtin} to this shifted equation, we obtain $\tilde X_\eps (t)=S^tz+\eps^\alpha \hat \phi_\eps (t)$, where, for $t>0$
\begin{align*}
\hat \phi_\eps (t) = \Phi_z (t) \phi_\eps (T) + \theta_\eps (t),
\end{align*}
and 
\[
\theta_\eps (t)= \eps^{1-\alpha}   \Phi_z (t) \int_0^t \Phi_z (s)^{-1}\sigma (S^s z)d\tilde W(s)  
+ \eps^{\alpha_1-\alpha}\Phi_z(t) \int_0^t \Phi_z (s)^{-1} \Psi_0( S^s z) ds + \tilde r_\eps (t),
\] where $\tilde r_\eps$ converges to 0 uniformly over compact time intervals in probability.
Then due to~\eqref{eqn: def_r_pm},
\begin{align*}
\tilde X_\eps (t) &=S^tz+\eps^\alpha (\Phi_z (t) \phi_\eps (t) + \theta_\eps (t) ) \\
&= z+tb(z) + r_+(t,z)+\eps^\alpha (\Phi_z (t) \phi_\eps (t) + \theta_\eps (t) ).
\end{align*}
Hence, with $\Gamma_{\eps,+} (t)=\theta_\eps (t) + \eps^{-\alpha} r_+(t,z)$ the result is a consequence of Lemma~\ref{lemma: rectification}.
\end{proof}

\bigskip

Let us now parametrize, locally around $z$, the hypersurface $M$ as a graph of a $C^2$-function $F$ over $T_zM$, i.e.,
$y\mapsto z+y+F(y)\cdot b(z)$ gives a $C^2$-parametrization of a neighborhood of $z$ in $M$ by a neighborhood of $0$ in $T_zM$. Moreover, $DF(0)=0$ so that
$|F(y)|=O(|y|^2)$, $y\to 0$.
With this definition, it is clear that, for $w\in \R^d$ with $w-z$ small enough, $w\in M$ if and only if $\pi_b(w-z)=F(\pi_M(w-z))$.

Let us define
\begin{align*} 
\Omega_{1,\eps}&=\bigl\{\tau_\eps = \inf \{  t\geq 0: \pi_b \left(X_\eps(t)-z \right) = F\left( \pi_M \left(X_\eps(t)-z \right)  \right)\}\bigr\},\\
\Omega_{2,\eps}&= \left \{  |\tau_\eps - T| \leq \eps^\gamma \right \},\\
\Omega_\eps&=\Omega_{1,\eps}\cap \Omega_{2,\eps}.
\end{align*} 

\begin{lemma}\label{lm:omega_eps_to_1} $\Pp(\Omega_{\eps})\to 1$ as $\eps\to 0$.
\end{lemma}

\begin{proof}
The definition of $F$ and Lemma~\ref{lemma: Bakhtin} imply that  as $\eps\to 0$, $\Pp(\Omega_{1,\eps})\to 1$.

We use~\eqref{eqn: X_after} to conclude that
\[
\pi_b \left( X_\eps (T+\eps^\gamma) -z \right)=\eps^\gamma  \left( 1+ \eps^{\alpha-\gamma} \pi_b \left( \Phi_z(\eps^\gamma) \phi_\eps(T) + \Gamma_{\eps,+} (\eps^\gamma) \right) \right), 
\]
and 
\[
F \left( \pi_M \left( X_\eps (T+\eps^\gamma) -z \right) \right)=F \left( \eps^\alpha \pi_M \left(  \Phi_z(\eps^\gamma) \phi_\eps(T) + \Gamma_{\eps,+} (\eps^\gamma) \right) \right).
\]
Since  $|F(x)|= O(|x|^2)$, these estimates imply that
\begin{multline*}
\limsup_{\eps \to 0} \Pp\left( \left\{  \tau_\eps > T+\eps^\gamma \right \}\cap \Omega_{1,\eps} \right)\\ \le \limsup_{\eps \to 0} \Pp \left\{  \pi_b \left( X_\eps (T+\eps^\gamma) -z \right) 
\le F \left( \pi_M \left( X_\eps (T+\eps^\gamma) -z \right) \right) \right \} =0.
\end{multline*}
It remains to prove
\begin{equation}
\lim_{\eps \to 0} \Pp \left\{  \tau_\eps <T-\eps^\gamma \right \}=0.
\label{eq:rough_lower_estimate_exit_time}
\end{equation}
Let us denote the Hausdorff distance between sets by $d(\cdot,\cdot)$. Then
an obvious estimate
\[
d(\{S^tx_0: 0\le t\le T-\delta\},M)\ge c\delta 
\]
 holds true for some $c>0$ and all sufficiently small $\delta>0$. Now~\eqref{eq:rough_lower_estimate_exit_time} follows from Lemma~\ref{lemma: Bakhtin},
and the proof is complete
\end{proof}

\begin{lemma} \label{prop: tau_conv}
Define $\tau_\eps^\prime = \tau_\eps - T $. Then,
\[
\eps^{-\alpha} \tau_\eps^\prime +  \pi_b \phi_\eps (T) \stackrel{\Pp}{\longrightarrow} 0, \quad \eps \to 0.
\]
\end{lemma}
\begin{proof}
 Let us define $A_\eps = \left \{  0 \leq \tau_\eps^\prime  \leq \eps^\gamma \right \}\cap\Omega_{1,\eps}$ and 
$B_\eps = \left \{  -\eps^\gamma \leq \tau_\eps^\prime  < 0 \right \}\cap\Omega_{1,\eps}$, so that 
$\Omega_\eps = A_\eps \cup B_\eps$.
We can use~\eqref{eqn: X_after} and the definition of $\Omega_{1,\eps}$ to get
\begin{align*}
\one_{A_\eps} \tau_\eps^\prime + \one_{A_\eps} \eps^\alpha  \pi_b \left( \Phi_z (\tau_\eps^\prime) \phi_\eps (T) +\Gamma_{\eps,+}(\tau_\eps^\prime) \right) = 
\one_{A_\eps}F\left( \eps^\alpha \pi_M \left( \Phi_z (\tau_\eps^\prime) \phi_\eps (T) +\Gamma_{\eps,+}(\tau_\eps^\prime) \right) \right).
\end{align*}
This implies
\begin{align} 
\notag
\one_{A_\eps} \eps^{-\alpha} \tau_\eps^\prime &=\eps^{-\alpha} \one_{A_\eps} F\left( \eps^\alpha \pi_M \left( \Phi_z (\tau_\eps^\prime) \phi_\eps (T) +\Gamma_{\eps,+}(\tau_\eps^\prime) \right) \right)
\\ \notag &\quad-\one_{A_\eps} \pi_b \left( \Phi_z (\tau_\eps^\prime) \phi_\eps (T) +\Gamma_{\eps,+}(\tau_\eps^\prime) \right) \\
\notag &= -\one_{A_\eps} \pi_b \left( \Phi_z (\tau_\eps^\prime) \phi_\eps (T)  \right) + r_{\eps,1} \\
&= -\one_{A_\eps} \pi_b  \phi_\eps (T)  +\one_{A_\eps} \pi_b \left( (I-\Phi_z (\tau_\eps^\prime) )\phi_\eps (T)  \right)+r_{\eps,1}, \label{eqn: tau_on_A}
\end{align}
where $r_{\eps,1}$ is a random variable that converges to $0$ in probability as $\eps \to 0$.

Likewise, since $\tau_\eps = T - (-\tau_\eps^\prime)$ and $\one_{B_\eps} \tau_\eps^\prime \le 0$, we can use ~\eqref{eqn: X_before} and the definition of $\Omega_{1,\eps}$ to see that 
\[
\one_{B_\eps}\tau_\eps^\prime + \one_{B_\eps}\eps^\alpha \pi_b \left ( \phi_\eps (T+\tau_\eps^\prime) + \Gamma_{\eps,-} (-\tau_\eps^\prime) \right ) = \one_{B_\eps} F \left (\eps^\alpha\left ( \phi_\eps (T+\tau_\eps^\prime) + \Gamma_{\eps,-} (-\tau_\eps^\prime) \right ) \right ).
\]
Hence, proceeding as before, we see that 
\begin{align*}
\one_{B_\eps} \eps^{-\alpha} \tau_\eps^\prime &= -\one_{B_\eps} \pi_b \phi_\eps (T+\tau_\eps^\prime) + r_{\eps,2} \\
&= -\one_{B_\eps} \pi_b \phi_\eps (T)+\one_{B_\eps} \pi_b \left( \phi_\eps (T)-\phi_\eps(T+\tau_\eps^\prime) \right)+r_{\eps,2}
\end{align*}
for some random variable $r_{\eps,2}$ such that $r_{\eps,2} \to 0$ in probability as $\eps \to 0$. Adding this identity and~\eqref{eqn: tau_on_A}, we see that on $\Omega_{\eps}$
\[
 \eps^{-\alpha} \tau_\eps^\prime=-\pi_b \phi_\eps (T)+\one_{A_\eps} \pi_b \left( ( I-\Phi_z (\tau_\eps^\prime) ) \phi_\eps (T)  \right)+\one_{B_\eps} \pi_b \left( \phi_\eps (T)-\phi_\eps(T+\tau_\eps^\prime) \right)+r_{\eps,1}+r_{\eps,2}.
\]
Due to Lemma~\ref{lm:omega_eps_to_1}, to finish the proof it is sufficient to notice that as $\eps \to 0$
\begin{equation}
\label{eqn: aux_1}
\sup_{0\leq t \leq \eps^\gamma } \left| (I-\Phi_z (t) ) \phi_\eps (T)  \right| \stackrel{\Pp}{\longrightarrow} 0,
\end{equation} and
\begin{equation} \label{eqn: aux_2}
\sup_{0\leq t \leq \eps^\gamma } \left| \phi_\eps (T)-\phi_\eps(T+t) \right|\stackrel{\Pp}{\longrightarrow} 0.
\end{equation}
\end{proof}

Lemma~\ref{prop: tau_conv} takes care of the time component in Theorem~\ref{thm: Main}. We shall consider the spatial component now.

 Let $A_\eps$ and $B_\eps$ be as in the proof of Lemma~\ref{prop: tau_conv}. Then, \eqref{eqn: X_after} implies
\begin{multline}
\one_{A_\eps} \left (  X_\eps (\tau_\eps) - z \right)\eps^{-\alpha} = \one_{A_\eps}\eps^{-\alpha} \tau_\eps^\prime b(z) + \one_{A_\eps} \left(  \Phi_z(\tau_\eps^\prime)\phi_\eps(T) + \Gamma_{\eps,+}(\tau_\eps^\prime) \right) \\
= \one_{A_\eps} \left( \eps^{-\alpha}\tau_\eps^\prime b(z) + \phi_\eps (T) \right) + \one_{A_\eps} \left[  (\Phi_z(\tau_\eps^\prime)- I )\phi_\eps(T) + \Gamma_{\eps,+}(\tau_\eps^\prime) \right] \label{eqn: spacial_after}
\end{multline}
Likewise, from~\eqref{eqn: X_before} we get that  
\begin{multline}
\one_{B_\eps} \left (  X_\eps (\tau_\eps) - z \right)\eps^{-\alpha} = \one_{B_\eps}\eps^{-\alpha} \tau_\eps^\prime b(z) + \one_{B_\eps} \left(  \phi_\eps(T+\tau_\eps^\prime) + \Gamma_{\eps,-}(-\tau_\eps^\prime) \right) \\
= \one_{B_\eps} \left(\eps^{-\alpha} \tau_\eps^\prime b(z) + \phi_\eps (T) \right) + \one_{B_\eps} \left[  (\phi_\eps(T+\tau_\eps^\prime)-\phi_\eps(T) ) + \Gamma_{\eps,-}(-\tau_\eps^\prime) \right] \label{eqn: spacial_before}.
\end{multline}
Adding~\eqref{eqn: spacial_after} and~\eqref{eqn: spacial_before} and proceding as in the proof of Lemma~\ref{prop: tau_conv} we see that
\[
\left (  X_\eps (\tau_\eps) - z \right)\eps^{-\alpha} - \pi_M \phi_\eps (T)= \left(\eps^{-\alpha} \tau_\eps^\prime +\pi_b \phi_\eps (T) \right)b(z) + \rho_\eps,
\] where, due to~\eqref{eqn: aux_1},~\eqref{eqn: aux_2} and Lemma~\ref{lemma: X_before_after}, $\rho_\eps \to 0$ in probability as $\eps \to 0$.  From this expression and Lemma~\ref{prop: tau_conv} we get that 
\[
\left (  X_\eps (\tau_\eps) - z \right)\eps^{-\alpha} - \pi_M \phi_\eps (T) \overset{\Pp}{\longrightarrow} 0, \quad \eps \to 0.
\]
Then, using this and the convergence in Lemma~\ref{prop: tau_conv}
\[
\eps^{-\alpha}(\tau_\eps-T, X_\eps(\tau_\eps)-z)=R_\eps + G(\phi_\eps (T) ),
\]
where $R_\eps$ is a random variable such that $R_\eps \to 0$ in probability as $\eps \to 0$. $G$ is the continuous function $x \mapsto (-\pi_b x, \pi_M x )$. Hence, Theorem~\ref{thm: Main} follows from the convergence in~Lemma~\ref{lemma: Bakhtin}.

It remains to prove Lemma~\ref{lemma: Bakhtin}, the core of the proof of the main result.

\begin{proof}[Proof of Lemma~\ref{lemma: Bakhtin}]
Let $\Delta_\eps^t= X_\eps (t) -S^t x_0$ and note that it satisfies the equation
\[
d\Delta_\eps^t= \left(  \left( b(X_\eps (t))-b(S^t x_0) \right) + \eps^{\alpha_1} \Psi_\eps(X_\eps (t)) \right)dt + \eps \sigma (X_\eps (t) ) dW(t),
\]
with initial condition $\Delta_\eps^0 = \eps^{\alpha_2} \xi_\eps$. We want to study the properties of this equation. We start with the difference in $b$.
Since $b$ is a $C^2$ vector field, we may write 
\begin{align} 
b(X_\eps (t))-b(S^t x_0) &=Db(S^t x_0) \Delta_\eps^t + Q_1(S^t x_0, \Delta_\eps^t).
 \label{eqn: expansionb}
\end{align}
Also, we can write 
\begin{equation} \label{eqn: expansion_Psi}
\Psi_\eps ( X_\eps (t) )=\Psi_0 ( S^t x_0)+Q_2 (S^t x_0,\Delta_\eps^t)+R_\eps(S^t x_0),
\end{equation} and
\begin{equation} \label{eqn: expansionSigma}
\sigma ( X_\eps (t) )=\sigma ( S^t x_0)+Q_3(S^t x_0,\Delta_\eps^t),
\end{equation}
where 
\[
 R_\eps(x)=\Psi_\eps(x)-\Psi_0(x)=o(1),\quad \eps\to 0,
\]
uniformly in $x$;
$Q_i:\R^d \times \R^d \to \R^d$, $i=2,3$ satisfies
\begin{equation}
|Q_i (u,v)|\leq K   |v|,\quad u,v\in\R^d  \label{eqn: bound_Q2}.
\end{equation}
We can assume that the constant $K>0$ in~\eqref{eqn: bound_Q1} and~\eqref{eqn: bound_Q2} 
is the same for simplicity of notation. 

Let $Q=Q_1+ \eps^{\alpha_1} Q_2+\eps^{\alpha_1}R_\eps$. Combine~\eqref{eqn: expansionb}, \eqref{eqn: expansion_Psi}, and \eqref{eqn: expansionSigma} to get 
\begin{align} \notag
d\Delta_\eps^t = &\left(  A(t)\Delta_\eps^t + \eps^{\alpha_1} \Psi_0 ( S^t x_0) + Q(S^t x_0,  \Delta_\eps^t) \right)dt \\
&\quad + \eps \left(  \sigma ( S^t x_0 ) +Q_3(S^t x_0,\Delta_\eps^t) \right) dW(t), \label{eqn: SDE_Delta} \\
\Delta_\eps^0 = &\eps^{\alpha_2} \xi_\eps.
\end{align}
Hence, applying Duhamel's principle to ~\eqref{eqn: SDE_Delta} and using~\eqref{eqn: SDE_Theta},  we get
\begin{align} \notag
\Delta_\eps^t & = \eps^\alpha \Theta_\eps (t)+ \Phi_{x_0} (t) \int_0^t \Phi_{x_0} (s)^{-1}Q(S^s x_0,  \Delta_\eps^s )ds\\ \notag
&\quad + \eps \Phi_{x_0} (t) \int_0^t \Phi_{x_0} (s)^{-1}Q_3(S^s x_0, \Delta_\eps^s )dW(s) \\ \label{eqn: expansion_X}
&= \eps^\alpha \Theta_\eps (t) + \Theta'_\eps (t),
\end{align} where $\Theta'_\eps$ is defined by~\eqref{eqn: expansion_X}.
A simple inspection of~\eqref{eqn: SDE_Theta} shows that $\left( \Theta_\eps \right)_{\eps>0}$ 
 converges in  distribution in $C(0,T)$ to the process $\phi_0 (t)$. This convergence is in probability if $\alpha_2 > {\alpha}$
or $\xi_\eps\to\xi_0$ in probability.
 Therefore, the lemma will follow with $\phi_\eps = \Theta_\eps + \eps^{-\alpha} \Theta'_\eps$ if we show that 
\begin{equation} \label{eqn: estimate_Theta'}
\eps^{-\alpha} \sup_{t \leq T} | \Theta'_\eps (t) | \overset{\Pp}{\longrightarrow} 0, \quad \eps \to 0.
\end{equation}
 For any $\delta \in (1/2,1)$, we introduce the stopping time
\[
l_\eps ( \delta )=\inf \left\{ t>0: |\Delta_\eps^t| \geq \eps^{\alpha \delta } \right \}.
\]
Now, $\Theta'_\eps=\Theta'_{\eps,1}+\eps \Theta'_{\eps,2}$, where
\[
\Theta'_{\eps,1} (t) = \Phi_{x_0} (t) \int_0^t \Phi_{x_0} (s)^{-1}Q(S^s x_0, \Delta_\eps^s)ds,
\] and
\[
\Theta'_{\eps,2} (t)=\eps \Phi_{x_0} (t) \int_0^t \Phi_{x_0} (s)^{-1}Q_3(S^s x_0,  \Delta_\eps^s )dW(s).
\]
Bounds~\eqref{eqn: bound_Q1}, and \eqref{eqn: bound_Q2} imply
\begin{equation}
\sup_{t\leq T\wedge l_\eps(\delta) }|\Theta'_{\eps,1} (t) | = O( \eps^{2 \alpha \delta}  + \eps^ {{\alpha_1}+\alpha \delta})+o(\eps^{\alpha_1})=o(\eps^{\alpha}). \label{eqn: Theta_prime_1}
\end{equation}
Likewise,~\eqref{eqn: bound_Q2} for $Q_3$ and BDG inequality imply that for any $\kappa>0$ there is a constant $K_\kappa$ such that
\begin{equation}
\Pp \left \{ \sup_{t\leq T\wedge l_\eps(\delta) }|\Theta'_{\eps,2} (t) | > K_\kappa \eps^ {1+\alpha\delta} \right \} < \kappa \label{eqn: Theta_prime_2}
\end{equation} for all $\eps>0$ small enough. Then, this together with~\eqref{eqn: Theta_prime_1} imply that 
\begin{equation}
\eps^{-\alpha\delta} \sup_{t \leq T \wedge l_\eps (\delta)} | \Theta'_\eps (t) | \overset{\Pp}{\longrightarrow} 0, \quad \eps \to 0.
\label{eq:theta-prime-small}
\end{equation}
Then, if $l_\eps (\delta) < T$ we use~\eqref{eqn: expansion_X} to get
\begin{align*}
1 &=\eps^{-\alpha\delta} \sup_{t\leq T \wedge l_\eps (\delta)} | \Delta_\eps^t | \\
& \leq \eps^{\alpha(1-\delta)} \sup_{t\leq  T \wedge l_\eps (\delta)} | \Theta_\eps (t) | +\eps^{-\alpha\delta} 
\sup_{t\leq  T \wedge l_\eps (\delta)}  | \Theta'_\eps (t) |.
\end{align*} The r.h.s. converges to 0 in probability due to~\eqref{eq:theta-prime-small} and the tightness of distributions of $\Theta_\eps$. Hence, $\Pp \{ l_\eps (\delta) < T \} \to~0$ as $\eps~\to~0$. Using $T$ instead of $T \wedge l_\eps ( \delta) $ in~\eqref{eqn: Theta_prime_1} and~\eqref{eqn: Theta_prime_2}, we see that with the choice of $\delta > 1/2$, \eqref{eqn: estimate_Theta'} follows and the proof is finished.
\end{proof}

\section{Proof of Lemmas~\ref{lemma: conditioned_drift} and~\ref{lemma: b_approx}}\label{sec:aux-1-d}

\begin{proof}[Proof of Lemma~\ref{lemma: conditioned_drift}] Let us find the generator of the conditioned diffusion. To that end we denote the generator of the original diffusion
by $L_\eps$:
\begin{equation}
\label{eq:generator1}
 L_\eps f(x)=b(x)f'(x)+\frac{\eps^2}{2}\sigma^2(x)f''(x)=\lim_{t\to 0}\frac{\E_x f(X_\eps)-f(x)}{t},
\end{equation}
where $f$ is any bounded $C^2$-function with bounded first two derivatives and $\E_x$ denotes expectation with respect to the
measure $\Pp_x$,
the element of the Markov family describing the Markov process emitted from point $x$.

Let us denote $u_\eps(x)=\Pp_x(B_\eps)$. This function solves the following boundary-value problem for the 
backward Kolmogorov equation:
\begin{align*}
L_\eps u_\eps(x)=0,\quad u_\eps(a_1)=0,\quad
u_\eps(a_2)=1. 
\end{align*}
Using~\eqref{eq:generator1}, it is easy to check that a unique solution is given by
\[
 u_\eps(x)=\frac{\int_{a_1}^xh_\eps(y)dy}{\int_{a_1}^{a_2}h_\eps(y)dy},
\]
where $h_\eps$ is defined in~\eqref{eq:h_eps}.

Now we can compute the generator $\bar L$ of the conditioned flow. For any smooth and bounded function $f\in C^2$ with
bounded first two derivatives, we can write
\begin{align*}
\E_x[f(X_\eps)|B_\eps]&=u^{-1}(x)\E_x f(X_\eps(t))\one_{B_\eps}
\\&=u_\eps^{-1}(x)\E_x f(X_\eps(t))\one_{B_\eps}\one_{\{\tau_\eps\ge t\}}+R_\eps
\\&=u_\eps^{-1}(x)\E_x\E_x [f(X_\eps(t))\one_{B_\eps}\one_{\{\tau_\eps\ge t\}}|\Fc_t]+R_\eps
\\&=u_\eps^{-1}(x)\E_x f(X_\eps(t)) \Pp_{X_\eps(t)}(B_\eps)+R_\eps
\\&=u_\eps^{-1}(x)\E_x f(X_\eps(t)) u_\eps(X_\eps(t))+R_\eps,
\end{align*}
where
\[
|R_\eps|=u_\eps^{-1}(x)|\E_x f(X_\eps)\one_{B_\eps}\one_{\{\tau_\eps<t\}}|\le C(x)\Pp\{\tau_\eps<t\}=o(t)
\]
for some $C(x)>0$. Therefore, we
obtain
\begin{align*}
\bar L f(x)&=\lim_{t\to 0} \frac{\E_x[f(X_\eps(t))|B_\eps]-f(x)}{t}
\\ &=\lim_{t\to 0}\frac{u_\eps^{-1}(x)\E_x f(X_\eps(t)) u_\eps(X_\eps(t))-f(x)}{t}
\\ &=\frac{1}{u_\eps(x)}\lim_{t\to 0}\frac{\E_x f(X_\eps(t)) u_\eps(X_\eps(t))-f(x)u_\eps(x)}{t}
\\ &= \frac{1}{u_\eps(x)} L(fu_\eps) (x)
\\ &= \left(b(x)+\eps^2\sigma^2(x)\frac{u'_\eps(x)}{u_\eps(x)}\right)f'(x)+\eps^2\frac{\sigma^2(x)}{2}f''(x).
\\ &= \left(b(x)+\eps^2\sigma^2(x)\frac{h_\eps(x)}{\int_{a_1}^x h_\eps(y)dy}\right)f'(x)+\eps^2\frac{\sigma^2(x)}{2}f''(x),
\end{align*}
completing the proof.
\end{proof}

\begin{proof}[Proof of Lemma~\ref{lemma: b_approx}] The proof is a variation of Laplace's method. 
Let 
\begin{equation}
\Phi (x)=2 \int_{a_1}^x \frac{ b(y)}{ \sigma^2 (y) } dy,\quad x\ge a_1, \label{eqn: def_Phi}
\end{equation}
so that $h_\eps (x) = e^{ -\Phi (x) / \eps^2 }$.  We take any $\beta \in (1,2)$ and break the integral of $h_\eps$ in two parts:
\[
\int_{a_1}^x e^{ -\Phi (y) /\eps^2 } dy = I_{\eps,1} (x) + I_{\eps,2} (x),
\]
where
\begin{equation} \label{eqn: I_eps,1}
I_{\eps,1} (x) =\int_{a_1}^{ x-\eps^\beta} e^{ -\Phi (y) /\eps^2 } dy,
\end{equation}
and 
\begin{equation} \label{eqn: I_eps,2}
I_{\eps,2} (x) =  \int_{ x-\eps^\beta} ^ x e^{ -\Phi (y) /\eps^2 } dy.
\end{equation}
The idea is to prove that $I_{\eps,1}$ is exponentially smaller than $I_{\eps,2}$ and then estimate~$I_{\eps,2}$. 

We start with some preliminaries for the function $\Phi$. Since both $b$ and $\sigma$ are $C^1$ and $\sigma \neq 0 $ in $[a_1, a_2 ]$ we conclude that $\Phi$ is a $C^2$ function so that  we can find a function $R:\R \times \R \to \R$ and a number $\delta_0>0$ such that for every $x,y \in [a_1,a_2+\delta_0]$, we have the expansion
\begin{equation} \label{eqn: Taylor_phi}
\Phi (y)= \Phi (x) + \Phi ^ \prime (x) (y -x) + R(x, y-x ),
\end{equation}
and 
\begin{equation}
|R(x,v)| \leq K_1 |v|^2,\quad x \in [a_1, a_2+\delta_0], v \in \R, \label{eqn: bound_R}
\end{equation}
for some $K_1>0$.

To estimate $I_{\eps,1}$, we introduce 
\[J_{\eps,1} (x) = \frac{e^{\Phi(x) / \eps^2}}{\eps^2\sigma^2(x)} I_{\eps, 1 } (x),\quad x \in [a_1, a_2+\delta_0].\]
Since $\Phi$ is decreasing, we have that for some constant $K_2>0$ independent of $x \in [a_1, a_2+\delta_0]$,
\begin{align} 
J_{\eps,1} (x) \leq \frac{K_2}{\eps^2}e^ { ( \Phi(x)-\Phi(x-\eps^\beta) )  / \eps^2}. \label{eqn: estimate_2}
\end{align}
Since $\beta<2$ and $\Phi^\prime$ is negative and bounded away from zero, we conclude that there is
$\alpha(\eps)$ such that $\alpha(\eps)=o (\eps ^ 2 )$ as $\eps\to 0$ and 
\begin{equation}
\label{eq:sup_estimate-on-J_1}
 \sup_{x \in [a_1, a_2+\delta_0]} J_{\eps,1} (x) \leq \alpha(\eps).
\end{equation}

We now estimate $I_{\eps,2}$. Using expansion~\eqref{eqn: Taylor_phi} and the change of variables $u=- \Phi(x) (y-x)/\eps^2$, we get
\begin{align} \notag
I_{\eps,2}(x)&= e^{ -\Phi (x) /\eps^2 } \int_{ x- \eps^\beta}^x e^{ -\Phi ^ \prime (x) (y -x)/\eps^2 - R(x, y-x ) /\eps^2 } dy \\ \notag
&=-\frac{  \eps^2  } {\Phi^\prime (x) } e^{ -\Phi (x) /\eps^2 }  \int_{\Phi^\prime(x)/\eps^{2-\beta}}^0 e^{ u - R(x,-\eps^2 u/ \Phi^\prime(x) ) /\eps^2 } du \\
&=-\frac{  \eps^2 \sigma^2 (x) } {2 b(x) } e^{ -\Phi (x) /\eps^2 } J_{\eps,2} (x), \label{eqn: Laplace_int}
\end{align}
where we use~\eqref{eqn: def_Phi} to compute the derivative of $\Phi$, and we define $J_{\eps,2}$ by~\eqref{eqn: Laplace_int}.
Hence, combining ~\eqref{eqn: estimate_2} with the definition of $b_\eps$ and~\eqref{eqn: Laplace_int}, we get 
\begin{align*}
b_\eps (x) = b(x) + \frac{1}{J_{\eps,1}(x) - \frac{1}{2 b(x)} J_{\eps,2} (x) } .
\end{align*}
Due to \eqref{eq:sup_estimate-on-J_1}, the proof will be complete once we prove that for sufficiently small $\delta>0$,
\[
\limsup_{\eps \to 0}  \eps^{-2}\left(\sup_{x \in [x_0-\delta, a_2+\delta]} | J_{\eps,2} (x) -1 | \right) < \infty.
\]

Note that for any $\delta\in(0,x_0-a_1)$, some constant $K_3=K_3(\delta)>0$ and all $x \in [x_0-\delta,a_2+\delta]$,
\begin{align} \notag
|J_{\eps,2} (x)-1| &=\Bigl| \int_{\Phi^\prime (x)/\eps^{2-\beta} }^0 e^u (1 - e^{ - R(x,-\eps^2 u/ \Phi^\prime(x) ) /\eps^2 } )du  \\ \notag
& \quad +\int_{-\infty}^{\Phi^\prime (x)/\eps^{2-\beta} } e^u du\Bigr| \\
&\leq  \int_{\Phi^\prime (x)/\eps^{2-\beta}}^0 e^u | 1 - e^{ - R(x,-\eps^2 u/ \Phi^\prime(x) ) /\eps^2 } |du  + e^{- K_3/\eps^{2-\beta} }. \label{eqn: approx_int}
\end{align}
Using~\eqref{eqn: bound_R} we see that for some constant $K_4>0$ independent of $x \in [x_0-\delta,a_2+\delta]$ and $u \in \R$, 
\[
| R(x,-\eps^2 u/ \Phi^\prime(x) ) |/\eps^2 \leq K_4 \eps^2 u^2.
\]
In particular, 
\[
\sup_{x \in [x_0-\delta,a_2+\delta]}\sup_{ u \in [\Phi^\prime (x)/\eps^{2-\beta} , 0] } | R(x,-\eps^2 u/ \Phi^\prime(x) ) |/\eps^2 \leq K_4 \eps^{2 ( \beta -1) }.
\]
Since $\beta>1$, the r.h.s.\ converges to $0$ and we can apply a basic Taylor estimate which implies that for all $\eps>0$ small enough,
\[
 \sup_{x \in [x_0-\delta,a_2+\delta]} \sup_{ u \in [\Phi^\prime (x)/\eps^{2-\beta} , 0] } | 1 - e^{ - R(x,-\eps^2 u/ \Phi^\prime(x) ) /\eps^2 } | \leq K_5 \eps^2 u^2,
\]
for some $K_5>0$. Using this fact in the integral of~\eqref{eqn: approx_int}, we can find a constant $K_6=K_6(\delta)>0$ such that
\[
\sup_{x \in [x_0-\delta, a_2+\delta]}|J_{\eps,2}(x)-1| \leq K_6 \eps^2 + e^{-K_3/\eps^{2-\beta}}, 
\]
which finishes the proof.
\end{proof}

\bibliographystyle{plain}
\bibliography{happydle}
\end{document}